\documentclass[11pt]{amsart}
\usepackage{amsfonts}
\usepackage{algorithmic}
\usepackage{mathrsfs}
\usepackage{graphicx}
\usepackage{verbatim}
\usepackage{caption}
\usepackage{array}

\usepackage{amsmath,amsthm,amssymb,latexsym}
\usepackage{color}
\usepackage[colorlinks,pagebackref=true]{hyperref}
\hypersetup{colorlinks,linkcolor={blue},citecolor={blue},urlcolor={red}} 

\newtheorem*{theorem*}{Theorem}
\newtheorem*{proposition*}{Proposition}
\newtheorem{theorem}{Theorem}[section]
\newtheorem{lemma}[theorem]{Lemma}
\newtheorem{corollary}[theorem]{Corollary}
\newtheorem{proposition}[theorem]{Proposition}

\theoremstyle{definition}
\newtheorem{definition}[theorem]{Definition}

\theoremstyle{remark}
\newtheorem{remark}[theorem]{Remark}

\numberwithin{equation}{section}

\author{Arpan Kabiraj}
\address{Chennai Mathematical Institute, Chennai, India}
\email{arpan.into@gmail.com}

\author{T. V. H. Prathamesh}
\address{Indian Statistical Institute, Chennai, India}
\email{prathamesh.t@gmail.com}

\author{Rishi Vyas}
\email{vyas.rishi@gmail.com}

\title[Elementary equivalence in Artin groups]{Elementary equivalence in Artin groups of finite type}


\begin{document}
\begin{abstract}
Irreducible Artin groups of finite type can be parametrized via their associated Coxeter diagrams into six sporadic examples and four infinite families, each of which is further parametrized by the natural numbers. Within each of these four infinite families, we investigate the relationship between elementary equivalence and isomorphism. For three out of the four families, we show that two groups in the same family are equivalent if and only if they are isomorphic; a positive, but weaker, result is also attained for the fourth family. In particular, we show that two braid groups are elementarily equivalent if and only if they are isomorphic. The  $(\forall\exists\forall)^{1}$ fragment suffices to distinguish the elementary theories of the groups in question.  

As a consequence of our work, we prove that there are infinitely many elementary equivalence classes of irreducible Artin groups of finite type. We also show that mapping class groups of closed surfaces - a geometric analogue of braid groups - are elementarily equivalent if and only if they are isomorphic.
\end{abstract}
\maketitle
\section*{Introduction}

Understanding when two non-isomorphic groups have distinct elementary theories has been a long-standing problem of interest in both group theory and model theory. In general, this problem is fairly difficult. Much of the current literature considers families of groups parametrised in a certain fashion, and attempts to determine to what extent these parameters are determined by the elementary theories of the groups in question. A particularly celebrated result, which follows from the work of of Sela \cite{Sela} and independently, Kharlampovich-Myasnikov \cite{KM}, is that the elementary theory of a non-abelian free group is independent of its rank: this resolved a famous question of Tarski. Some other important classes of groups for which something is known are listed below (this list is not meant to be exhaustive):
\begin{enumerate}
	\item  Finitely generated free abelian groups, by W. Szmielew \cite{Szw} (1955).
    \item  Ordered abelian groups, by A. Robinson and E. Zakon \cite{Rob} (1960), M. Kargapolov \cite{Kag} (1963) and Y. Gurevich \cite{Gur} (1964).
    \item  Classical linear groups, by A. Maltsev \cite{Mal} (1961).
    \item  Linear groups over the integers, by V. Durlev \cite{Dur} (1995).
    \item  Some linear and algebraic groups, by E. Bunina and A. Mikhal\"{e}v \cite{BunM} (2000).
    \item  Chevalley groups, by E. Bunina \cite{Bun} (2001).
    \item Right-angled Coxeter groups and graph products of finite abelian groups, by M. Casals-Ruiz, I. Kazachkov and V. Remeslennikov \cite{MKR} (2008).
\end{enumerate} 

For a detailed survey of what is known about the elementary theory of various classes of groups see \cite{BM} (also see \cite{MKR}). By studying examples of groups with different elementary theories, we can gain insight into the nature of first order statements in group theory.  

In this paper, we are primarily concerned with irreducible Artin groups of finite type: archetypical examples of such groups are braid groups. To every Coxeter matrix $C$ we can associate two groups: the Artin group $G_{C}$ and the Coxeter group $\bar{G}_{C}$. The Artin group $G_{C}$ is said to be of finite type if $\bar{G}_{C}$ is finite. The Coxeter diagrams corresponding to irreducible Artin groups of finite type have been completely classified, and can be organized into four infinite families (indexed by the natural numbers, so $A_{n}$, $B_{n}$, $D_{n}$ and $I_{2}(n)$) and six sporadic examples (the reader should note that for very small $n$, isomorphism classes of groups in the four infinite families may overlap). Details will be provided in Section \ref{art-groups}.

In this paper, we study elementary equivalence classes within (not between!)\footnote{The proofs of Lemma \ref{main lemma} and Theorem \ref{main theorem}, however, are strong enough to distinguish some groups between these classes.}  these four families. Our main result, Theorem \ref{main theorem}, is that within three out of these four families ($A_{n}$, $B_{n}$ and $D_{n}$), elementary equivalence class determines isomorphism class, and thus the parameter $n$. For the family $I_{2}(n)$ a weaker statement is attained. An immediate consequence of our work here is that there are infinitely many classes of elementary theories amongst Artin groups of finite type. Moreover, we show that all the above results hold true within the $(\forall\exists\forall)^1$ fragment of the elementary theory. 

Braid groups are examples of irreducible Artin groups of finite type of particularly significant interest. The following result follows as a corollary of Theorem \ref{main theorem}:
\begin{theorem*} 
Any two braid groups are elementarily equivalent if and only if they are isomorphic.
\end{theorem*}

We prove the above results by explicitly constructing a class of first-order sentences $\{\Phi_n\}_{n\in \mathbb{N}}$ to help us distinguish elementary theories; $\Phi_n$ expresses the notion that every central element has an $n^{\mathrm{th}}$ root. 

Irreducible Artin groups of finite type can be treated as algebraic generalizations of braid groups. The natural generalization in terms of geometric group theory would be the mapping class groups, as braid groups occur as mapping class groups of punctured discs. Mapping class groups of surfaces with non-empty boundary and punctures  along with mapping class groups of closed surfaces are two of the most significant classes of  groups in geometric group theory. Using a result \cite[Theorem 7.5]{PMCG} about cyclic subgroups of such groups, it becomes straightforward to prove the following result:

\begin{proposition*}
Let $Mod(S_g)$ denote the mapping class group of a closed surface $S_g$ of genus $g$. $Mod(S_g)$ is elementarily equivalent to $Mod(S_h)$ if and only if $g=h$.  
\end{proposition*}

This paper is organized along the following lines. Section \ref{two} contains a preliminary introduction to the first-order theory of groups; we also provide a short proof of the above proposition on elementary equivalence in the context of mapping class groups of closed surfaces. Section \ref{art-groups} is an overview of the theory of Artin groups of finite type: we list the basic definitions and results that will be used over the course of this document. A reader familiar with the theory of Artin groups may skip reading this section in detail, but we still recommend they take a quick glance in order to familiarise themselves with the notation used. Section \ref{art-mon} and \ref{art-rot} contain some key lemmas about Artin monoids and roots of central elements in irreducible Artin groups of finite type. Section \ref{elem} contains our main results concerning elementary equivalence in irreducible Artin groups of finite type.  \\

\textbf{Acknowledgements}: All three authors would like to thank the Institute of Mathematical Sciences, Chennnai and the Chennai Mathematical Institute for their support and hospitality. The first author is supported by the Department of Science \& Technology (DST): INSPIRE Faculty. The second author would like to thank Siddhartha Gadgil and Igor Rivin for (independently) suggesting the question considered here in the context of braid groups. 

\section{Logical Preliminaries and Mapping Class Groups}\label{two}

\subsection{First Order Logic}

This section contains a very brief introduction to first order logic in the context of group theory. It contains only those definitions which are pertinent to our work and context. For a broader and more comprehensive introduction to first-order logic, the reader is referred to \cite{HW}.

The first-order language of groups $\mathcal{L}_G$ is the tuple $(\cdot,\ ^{-1},\ 1)$, where $\cdot$ refers to the multiplication, $^{-1}$  is the multiplicative inverse and $1$ is the multiplicative identity.

An \textit{atomic formula} with variables $x_1,\ldots,x_n$ in $\mathcal{L}_G$  is a  statement of the form: 
\[{x_1}^{\varepsilon_1} \cdot {x_2}^{\varepsilon_2}\cdot \ldots \cdot {x_n}^{\varepsilon_n} =1,\]
where each  $\varepsilon_i \in \{\pm  1 \}$.  A \textit{quantifier free formula} in $\mathcal{L}_G$ is recursively defined as either an atomic formula, the negation of a quantifier free formula, the conjunction of finitely many quantifier free formulas, or a disjunction of finitely many quantifier free formulas. 

A \textit{sentence} in $\mathcal{L}_G$ is a  statement of the following form:
\[ Q_1 x_1. Q_2 x_1 \ldots Q_n x_n. \big( \Phi(x_1, x_2, x_3,\ldots ,x_n) \big), \]
where each $Q_i \in \{ \exists, \forall\}$, and $\Phi(x_1, x_2,\ldots , x_n)$ is a quantifier free formula with variables $x_1$, $x_2,\ldots ,x_n$. 

The set of all sentences which are hold in the group $G$ is called the \textit{elementary theory} of $G$. It is denoted by $\operatorname{Th}(G)$.

The elementary theory of a finite group determines the group up to isomorphism. This is no longer true for infinite groups: for instance, all finitely generated free groups have the same elementary theory (see  \cite{Sela}, \cite{KM}).

The class of sentences in the first-order language of groups is strong enough to describe admission of roots of central elements in a group. Consider the following statement:
\[\Phi_n =  \forall x. \exists y. \forall z. (\neg (xz = zx) \vee (x = y^n)).\]
$\Phi_n$ is true in a group $G$ precisely when every central element admits an $n^{\mathrm{th}}$ root. 

One can also describe the existence of a finite cyclic subgroup of order $n$ in a group by the following sentence: 
\[\Psi_n =  \exists x. ((x^n = 1) \wedge_{k=1}^{n-1} (x^k \neq 1)).\]

A sentence is of the class $(\forall \exists \forall)^{1}$ if it is of the form $\forall x \exists y\forall z(\Psi (x,y,z))$. This is a well studied class of sentences called the\textit{ Kahr class} (see Chapter 3.1 of \cite{BGG}).

\subsection{Mapping Class Groups}\label{mcg}
In this subsection, we discuss elementary equivalence in the context of mapping class groups of closed surfaces: using a result from the literature, we are able to provide a short proof of the fact that elementary equivalence determines isomorphism class. The reader may consider the material here as motivation for our results on Artin groups of finite type - it is an easy example of how explicit first order sentences can be used to distinguish elementary equivalence classes.

Let $S_g$ be a closed orientable surface of genus $g\geq 2$. The mapping class group $Mod(S_g)$ of the surface $S_g$ is the group of all homotopy classes of orientation preserving homeomorphisms of $S_g$. In this section we show that if $g\neq h$, then the elementary theories of $Mod(S_g)$ and $Mod(S_h)$ are not equivalent; indeed, the $(\exists)^{1}$ fragment of the elementary theory suffices to distinguish the elementary theories in question. 

To prove this result we need the following theorem about finite cyclic subgroups of mapping class groups. The second part of the theorem follows from the statement after Theorem 7.5 in \cite{PMCG}.

\begin{theorem}{\cite[Theorem 7.5]{PMCG}}\label{tormcg}
	The order of a finite cyclic subgroup of the mapping class group  $Mod(S_g)$ is at most $4g+2$. Moreover for every $g\geq 2$, there exists an element of order $4g+2$ in the mapping class group $Mod(S_g)$.
\end{theorem}  

\begin{proposition}
	The elementary theories of $Mod(S_g)$ and $Mod(S_h)$ are distinct for $g\neq h$.
\end{proposition}
\begin{proof}
	Consider the following first order statement statement:$$\Psi_{4g+2} = \exists x. (\left(x^{4g+2}=1\right) \wedge \left(x^{k}\neq 1 \,\, \text{for}\,\, k\in\{1,2,\ldots , 4g+1\}\right)).$$
	By Theorem \ref{tormcg}, the above statement is true in $Mod(S_g)$ but false in $Mod(S_h)$ for $g\neq h$.
\end{proof}

\section{Artin Groups} \label{art-groups}

We recall the definition and basic algebraic properties of Artin groups of finite type and Coxeter groups. For detailed exposition and proofs of the results mentioned here see \cite{Artin-Cox} and \cite[\S 1, Chapter IV]{Bourbaki}.

Suppose $C=(m_{i,j})$ denotes a $n\times n$  symmetric matrix with $(i,j)^{\mathrm{th}}$ entry $m_{i,j},$  where $m_{i,i}=1$ and $m_{i,j}\in \{2,3,\ldots ,\infty\}$ for $i\neq j$. Such a matrix is called a \emph{Coxeter matrix}. 

To every Coxeter matrix $C$ we can associate a labelled graph. If $C$ is an $n\times n$-matrix its associated graph has $n$ labelled ordered vertices, say $x_{1},\ldots,x_{n}.$ There is an edge between $x_{i}$ and $x_{j}$ with label $m_{i,j}$ if and only if $m_{i,j}\geq 3$; it is a convention to drop the label if $m_{i,j} = 3$. Such a labelled graph is called a \emph{Coxeter graph} or \emph{Coxeter diagram}. Coxeter matrices are in a canonical one-to-one correspondence with Coxeter diagrams. We will treat these two notions interchangeably in this paper. 

Let $\langle x,y\rangle^m$ denote the alternating product of $x$ and $y$ of length $m$ starting with $x$ (e.g. $
\langle x,y\rangle^3=xyx$). By convention, $\langle x,y\rangle^{\infty}$ is the empty product. 

\begin{definition} \label{defnartgroup}
Let $C$ be a Coxeter matrix with $(i,j)^{\mathrm{th}}$ entry $m_{i,j}$. The Artin group corresponding to $C$, $G_{C}$, is the group presented by the following presentation:
	\begin{align*}
	\Big\langle x_1,x_2,\ldots,x_n\, \Big| \, \langle x_i, x_j \rangle^{m_{i,j}}  =\langle x_j, x_i \rangle^{m_{j,i}},\, i,j\in \{1,\ldots, n\} \Big\rangle.
	\end{align*}
\end{definition}

The presentation used in the definition above is called the \textit{standard presentation} of an Artin group $G_{C}$. We shall from here on assume that unless stated otherwise, the terms presentation, generators and relations used in the context of an Artin group refer to the standard presentation and the associated generators and relations respectively. 

\begin{definition} \label{defncoxgroup}
Let $C$ be a Coxeter matrix with $(i,j)^{\mathrm{th}}$ entry $m_{i,j}$. The Coxeter group corresponding to $C$, $\bar{G}_{C}$, is the group presented by the following presentation:
	\begin{align*}
	\Big\langle x_1,x_2,\ldots,x_n\, \Big| \, (x_i x_j)^{m_{i,j}}  =1,\,
	i,j\in \{1,\ldots, n\}\,\, \text{and}\,\,m_{i,j}\neq\infty \Big\rangle.
	\end{align*}
\end{definition}

As in the Artin group case, we call the above presentation the standard presentation of a Coxeter group. The terms presentation, generators, and relations used in the context of a Coxeter group will refer to the standard ones. It is straightforward to check that a Coxeter group $\bar{G}_{C}$ with generators $x_1,x_2,\ldots ,x_n$ is the quotient of the Artin group with the same set of generators by the relation $x_i^2=1$ for all $i\in \{1,2,\ldots ,n\}$. 

The Coxeter diagram $C$ can be recovered from the Coxeter group $\bar{G}_{C}$  \emph{and} its standard generators. The vertices of the graph correspond to the generators, and $m_{i,j}$ can be recovered from the order of $x_{i}x_{j}$ in $\bar{G}_{C}$. 

An Artin group $G_{C}$ is said to be of \emph{finite type} if the Coxeter group $\bar{G}_{C}$ associated to $C$ is finite. An Artin group $G_{C}$ and the corresponding Coxeter group $\bar{G}_{C}$ are called \emph{irreducible} if the associated Coxeter diagram $C$ is connected. Throughout this paper we assume all Artin groups to be irreducible and of finite type unless otherwise mentioned, though we may mention this hypothesis explicitly on occasion for the sake of clarity. 

\begin{figure}[h] \label{centergen}
	\centering
	\includegraphics[trim = 5mm 220mm 20mm 10mm, clip, width=12.5cm]{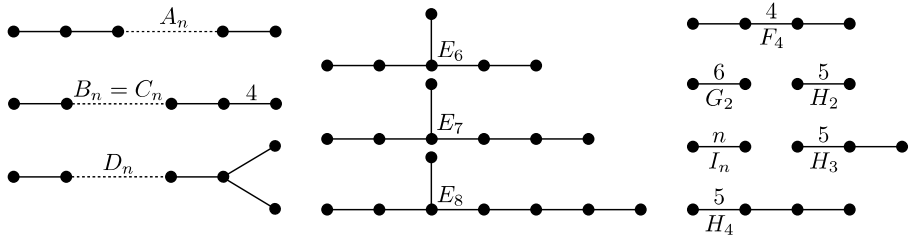}
	\caption{Coxeter diagrams corresponding to irreducible Coxeter groups of finite type. (P.C.-Wikimedia Commons.)}\label{coxfinite}
\end{figure}

Coxeter classified all irreducible Artin groups of finite type. In this case the Coxeter diagram is always a tree. There are four infinite families $A_n,B_n=C_n,D_n,$ and $I_2(n)$, and six distinct groups $E_6$, $E_7$, $E_8$, $F_4$, $H_3$, $H_4$ (Figure \ref{coxfinite}). In this paper, we will restrict our attention to groups within the four infinite families. We will denote the Artin groups associated to these families by the same notation i.e. by $A_{n}$, not $G_{A_{n}}$. 

When the diagram $C$ is either clear from the context or irrelevant, we may suppress it and denote the Artin group by $G$ and the corresponding Coxeter group by $\bar{G}$. However, we are obliged to remind the reader once more: in this paper, an Artin group means a group presented by a presentation associated to one of the diagrams in Figure \ref{coxfinite} \textit{along} with the data of its presentation. 

Let $G$ be an irreducible Artin group of finite type with generating set $I=\{x_1,x_2,\ldots , x_n\}$. Up to ordering, there is a unique partition of $I$ into two maximal disjoint subsets $J_1$ and $J_2$ such that the elements in $J_1$ (respectively $J_2$) commute pairwise in $\bar{G}$. Let 
$$\mathcal{J}_1=\prod_{x_i\in J_1}x_i,\hspace*{3mm} \mathcal{J}_2=\prod_{x_j\in J_2}x_j\hspace*{3mm}\text{and}\hspace*{3mm} \mathcal{J}=\mathcal{J}_1\mathcal{J}_2.$$ 
For every irreducible Artin group of finite type, there exists a corresponding natural number called the \textit{Coxeter number}. In Table \ref{numerical} below, we list the  Coxeter numbers associated to some Artin groups of finite type; for more information about this number see \cite[\S 6, Chapter IV]{Bourbaki}. Let $h$ be the Coxeter number of $G$, and define

\begin{equation} \label{Garsideelement}
\Delta:= 
\begin{cases}
\mathcal{J}^{\frac{h}{2}}& \text{ if } h \text{ is even, }\\
\mathcal{J}^{\frac{h-1}{2}}\mathcal{J}_1=\mathcal{J}_2\mathcal{J}^{\frac{h-1}{2}}             & \text{ if } h \text{ is odd. }
\end{cases}
\end{equation}

The following theorem follows from \cite[Lemma 5.8]{Artin-Cox} and the proposition following it in \textit{loc.~ cit.}

\begin{theorem} \label{delta-roots}
	For any irreducible Artin group of finite type $G$ with Coxeter number $h$ we have $\Delta^2=\mathcal{J}^h$. Furthermore, if $\Delta$ is in the center of $G$ then $
	\Delta=\mathcal{J}^{\frac{h}{2}}$.
\end{theorem}

\begin{remark} If $\Delta$ is in the center of an irreducible Artin group of finite type then $h$ is necessarily even: see Table \ref{numerical}.
\end{remark}
 
Indeed, we can say even more. The center of an irreducible Artin group of finite type $G$, $\operatorname{Z}(G)$, is cyclic and is generated by either $\Delta$ or $\Delta^2$. Moreover, we know exactly which of these two elements generates the centre in each of the cases that we care about. We will always refer to this choice of generator of $\operatorname{Z}(G)$ by $c_{G}$.  The following  theorem follows from the Corollary at the end of Section 7 of \cite{Artin-Cox}.
\begin{theorem}
		The center of an irreducible Artin group $G$ of finite type is infinite cyclic. 
		\begin{itemize}
			\item[(1)] For $B_n$, $D_{2n}$ and $I_2(2n)$ the center is generated by $\Delta$.
			\item[(2)] For $A_n$, $D_{2n+1}$ and $I_2(2n+1)$ the center is generated by $\Delta^2$
		\end{itemize}
\end{theorem}

In the following table we collect the numerics associated to the irreducible Artin groups of finite type required for our calculations: 

\begin{center} \label{numerical}
	\begin{tabular}{ | m{2cm} | m{2cm} | m{2cm} | m{2cm} | m{2cm} | } 
		\hline
		Group & Rank & Coxeter number ($h$) & Generator of the center ($c_{G}$) & Word length of $c_{G}$ \\
		\hline \hline
		$A_{k}$ & $k$ & $k+1$ & $\Delta^2$ & $k^2+k$ \\ 
		\hline
		$B_k$ & $k$ & $2k$ & $\Delta$ & $k^2$\\
		\hline
		$D_{2k+1}$ & $2k+1$ & $4k$ & $\Delta^2$ & $8k^2+4k$ \\
		\hline
		$D_{2k}$ & $2k$ & $4k-2$ & $\Delta$ & $4k^2-2k$\\
		\hline
		$I_2(2k+1)$ & $2$ & $2k+1$ & $\Delta^2$ & $4k+2$\\
		\hline
		$I_2(2k)$ & $2$ & $2k$ & $\Delta$ & $2k$\\
		\hline
	\end{tabular}
\end{center}

\begin{center}
Table \ref{numerical}.
\end{center}

\begin{remark} Observe that irrespective of whether $n$ is odd or even, the Coxeter numbers corresponding to $D_n$ and $I_2(n)$ are $2n-2$ and $n$ respectively. 
\end{remark}

The following lemma is a straightforward consequence of Theorem \ref{delta-roots}. 

\begin{lemma} \label{center-root}
Let $G$ be an irreducible Artin group of finite type. Let $c_{G}$ denotes the generator of the center of $G$. There is a word containing all the generators of $G$ which is equal to $c_{G}$ in $G$. Furthermore,
\begin{itemize}
	\item[(1)] $c_{A_{n}}$ admits an ${(n+1)}^{\mathrm{th}}$ root in $A_{n}$.
	\item[(2)] $c_{B_{n}}$ admits an $n^{\mathrm{th}}$ root in $B_{n}$.
	\item[(3)] If $n$ is odd, $c_{D_{n}}$ admits an ${(2n-2)}^{\mathrm{th}}$ root in $D_{n}$.
	\item[(4)] If $n$ is even, $c_{D_{n}}$ admits an ${(n-1)}^{\mathrm{th}}$ root in $D_{n}$.
	\item[(5)] If $n$ is odd, $c_{I_{2}(n)}$ admits an $n^{\mathrm{th}}$ root in $I_{2}(n)$.
	\item[(6)] If $n$ is even, $c_{I_{2}(n)}$ admits an ${(\frac {n} {2})}^{\mathrm{th}}$ root in $I_{2}(n)$.
	\end{itemize}
 \end{lemma}
 
\section{Key lemmas}
This section contains  results which we will later use in the proof of Theorem \ref{main theorem}. Section \ref{art-mon} includes material about the shape of words in Artin monoids. Section \ref{art-rot} contains the technical heart of this paper: Lemma \ref{main lemma}.

\subsection{Artin Monoids}\label{art-mon}

Let $G$ be an irreducible Artin group of finite type. Denote the monoid of positive words (with respect to the standard presentation) in $G$ by $G^{+}$. We call this monoid the Artin monoid.

For the benefit of the reader (and the authors!), we briefly recall the notion of a monoid presentation. Consider the presentation $P = \langle g_{1},\ldots, g_{n}\ \vert \ r_{1}=r'_{1},\ldots, r_{m}=r'_{m} \rangle$ where $r_{i}$ and $r'_{i}$ are positive words in the $g_{j}$. Consider the free monoid on the letters $g_{j}$; denote this by $F$. The monoid associated to the presentation $P$ is the quotient of $F$ by the smallest equivalence relation containing the relation $\{(xr_{i}y, xr'_{i}y)\}_{x,y\in F, i\in \{1,m\}}$; the set of equivalence classes clearly carries a natural monoid structure. 

\begin{proposition}\label{contgen}
Let $G$ be an irreducible Artin group of finite type. Let $w$ and $w'$ be positive words in the generators of $G$. Suppose $w=w'$ in $G^{+}$. If a generator $x_{i}$ appears in $w$, then it must appear in $w'$ as well. 
\end{proposition}
\begin{proof}
By \cite[Proposition 5.5]{Artin-Cox}, the Artin monoid $G^{+}$ is isomorphic to the \textit{monoid} presented by the Artin presentation associated to $G$ via the obvious isomorphism. 	

If $x_{i}$ appears as a letter on one side of any of the relations defining $G^{+}$, it appears on the other side as well. From this, and the definition of a monoid presentation, the result follows. 
\end{proof}

The following lemma is a direct consequence of the above proposition and the definition of $\Delta$. 
\begin{lemma}\label{center-contgen} 
Let $G$ be an irreducible Artin group of finite type. Every generator of $G$ appears in every positive word representing $\Delta$ or $\Delta^2$.
\end{lemma}

We also have the following lemma about the appearance of generators in powers of words.

\begin{lemma}\label{power-contgen}  
 Let $G$ be an irreducible Artin group of finite type. Let $x \in {G}^{+}$. If for some $n \in \mathbb{N}$ a generator $x_i$ appears in $x^n$, then it also appears in $x$.
\end{lemma}
\begin{proof}
Assume that a generator $x_i$ does not appear in a word representing $x$. Then, it does not appear in a word representing $x^k$. This contradicts Proposition \ref{contgen}.
\end{proof}

\subsection{Roots in Artin groups of finite type}\label{art-rot}

Artin groups of finite type are, in particular, examples of Garside groups. As there are many good references for theory of Garside groups we will not define this notion here, instead directing the reader towards \cite{Deh-P} (where these objects were first introduced) for a comprehensive overview. Briefly, however: a Garside group is a group that can be realized as the group of fractions of a Garside monoid. Garside monoids, in turn, are a class of cancellative monoids with good divisibility properties (\cite[Section 2]{Deh-P}). 

The fact that Artin groups of finite type are examples of Garside groups was essentially proved by Brieskorn and Saito in \cite{Artin-Cox} (also see \cite[Example 1, Section 2]{Deh-P}). More precisely, Brieskorn and Saito prove the following: suppose $C$ is a Coxeter matrix such that $\bar{G}_{C}$ is finite. Then, the Artin monoid $G_{C}^{+}$ is a Garside monoid with group of fractions $G_{C}$. 

In this paper, we are concerned with root extraction in Artin groups of finite type. An algorithm for extracting roots in Garside groups was given by Siebert in \cite[Algorithm 2.12]{sibgar}. It lies at the core of our proof of the following proposition. 

\begin{proposition}\label{positive root}
Let $G$ be an irreducible Artin group of finite type. Suppose $a\in \mathrm{Z}(G)\cap G^{+}$, and $n\in \mathbb{N}$. Then, the equation $x^{n}=a$ has a solution in $G$ if and only if it has a solution in $G^{+}$. 
\end{proposition}
\begin{proof}
This follows from the root extraction algorithm for Garside groups - see \cite[Algorithm 2.12]{sibgar}. Understanding how this algorithm works in our particular context is particularly easy because of our hypotheses on $a$: the fact that $a\in G^{+}$ and $a\in \mathrm{Z}(G)$ leads to a substantial simplification. 

However, there is one point that we wish to elaborate. The hypothesis in \cite[Algorithm 2.12]{sibgar} requires the Garside group $G$ to be the group of fractions of a Garside monoid $M$ with finite positive conjugacy classes. This hypothesis is satisfied in our case, by \cite [Corollary 2.4]{sibtame} and \cite[Criterion B]{sibtame}. 
\end{proof}

\begin{lemma}\label{main lemma}
Let $G$ be an irreducible Artin group of finite type. Let $c_{G}$ denote the generator of the center of $G$. The following statements hold:
\begin{enumerate}
    \item[(1)] For $k> n+1$, $c_{A_n}$ does not have a $k^{\mathrm{th}}$ root in $A_n$. 
	\item[(2)] For $k> n$, $c_{B_n}$ does not have a $k^{\mathrm{th}}$ root in $B_n$.
	\item[(3)] For $D_n$, the following holds:
		\begin{enumerate}
		\item[(a)] If $n$ is even and $k> n-1$, $c_{D_n}$ does not have a $k^{\mathrm{th}}$ root in $D_n$.
		\item[(b)] If $n$ is odd and $k> 2n-2$, $c_{D_n}$ does not have a $k^{\mathrm{th}}$ root in $D_n$.
		\end{enumerate}
	\item[(4)] For $I_2(n)$, the following holds:
		\begin{enumerate}
		\item[(a)] If $n$ is even and $k > {\frac {n} {2}}$, $c_{I_{2}(n)}$ does not have a $k^{\mathrm{th}}$ root in $I_{2}(n)$. 
		\item[(b)] If $n$ is odd and $k>n$, $c_{I_2(n)}$ does not have $k^{\mathrm{th}}$ root in $I_2(n)$.
		\end{enumerate}
\end{enumerate}
\end{lemma}

\begin{proof}
Every Artin group $G$ admits a group homomorphism $\lambda: G\to \mathbb{Z}$ defined by sending each generator in the standard presentation of $G$ to $1$; a moments thought will convince the reader that this map is well-defined. $\lambda$ restricts to a monoid homomorphism $\lambda: G^{+}\to \mathbb{N}$. We will call the restriction of $\lambda$ to $G^{+}$ the \textit{word length} function on $G^{+}$.
	
 We will use $\Delta_{n}$ for the element from \ref{Garsideelement} in the groups $A_{n}, B_{n}, D_{n},$ and $I_{2}(n)$: the group we are working in should be clear from the context.  
	
	(1): See Table \ref{numerical}: the center of $A_n$ is generated by $\Delta_n^2$. Suppose $\Delta_{n}^{2}$ has a $k^{\mathrm{th}}$ root for $k> n+1$.  By Proposition \ref{positive root}, there therefore exists $x\in A_n^+$ such that $x^{k}=\Delta_n^2$. Evaluating both sides by $\lambda$, we have $$k\lambda(x)=\lambda(\Delta_n^2)=n(n+1).$$ As $k>n+1$, this implies that $\lambda(x)<n$. 
	
	We know from Lemma \ref{center-contgen} that each $x_i$ appears in $\Delta_n^2$. Since $x^k = \Delta_{n}^2$,  as a consequence of Lemma \ref{contgen} and Lemma \ref{power-contgen} each $x_i$ must also appear in any word representing $x$. Since $x$ is a positive word in which each and every generator appears, $\lambda(x)\geq n$, a contradiction.  
	
	(2): See Table \ref{numerical}: the center of $B_n$ is generated by $\Delta_n$. Suppose $\Delta_{n}$ has a $k^{\mathrm{th}}$ root for $k>n$. Again by Proposition \ref{positive root}, there exists $x\in B_n^+$ such that $x^{k}=\Delta_n$. Evaluating both sides by $\lambda$, we see that $$k\lambda(x)=\lambda(\Delta_n)=n^2.$$ As $k>n$, we have $\lambda(x)<n$. But $x$ is a root of $\Delta_n$. Arguing as in the previous case, Lemma \ref{contgen}, Lemma \ref{center-contgen} and Lemma \ref{power-contgen} tell us that $\lambda(x)\geq n$. This is a contradiction.
	
	(3): In this case we need to be slightly more careful. From Table \ref{numerical}, we see that $\operatorname{Z}(D_{n})$ is generated by $\Delta_n$ when $n$ is even and by $\Delta_n^2$ when $n$ is odd. We also have 
    $$ 
    \begin{cases}
    \lambda(\Delta_n)=n(n-1) & \text{ if } n \text{ is even, }\\
    \lambda(\Delta_n^2)=n(2n-2) & \text{ if } n \text{ is odd. }
    \end{cases}
    $$ 	
	
	We consider the cases when $n$ is odd and even separately. 
	
	(a) {\textit{${\mathit{n}}$ is even}}: From Lemma \ref{positive root}, if there exists an $k^{\mathrm{th}}$ root for $c_{D_{n}}$ in $D_n$ there is one in ${D_n}^{+}$. Assume that $x^{k} = \Delta_n$, for some $k> n-1$ and $x\in {D_n}^{+}$. Analogous to earlier cases, we have the following equality:
	$$k\lambda(x)=\lambda(\Delta_n)=n(n-1).$$
     It thus follows that $\lambda(x) < n$. By arguments similar to what we have done above, we derive a contradiction.    
	 
	 (b) {\textit{${\mathit{n}}$ is odd}}: Let $k> 2n-2$.  As before, we assume the existence of a $k^{\mathrm{th}}$ root for the generator of the center, and thus the existence of a root in the positive monoid $D_{n}^{+}$, which we denote by $x$. As earlier we have the following equality
	 $$k\lambda (x)=\lambda (\Delta_n^2)=n(2n-2).$$
	  Thus it follows that $\lambda(x) < n$. This is a contradiction.
	 	 
	(4):  Gaze once more upon Table \ref{numerical}: $\operatorname{Z}(I_{2}(n))$ is generated by $\Delta_n$ when $n$ is even and by $\Delta_n^2$ when $n$ is odd. We also have 
	$$ 
	\begin{cases}
	\lambda(\Delta_n)=n & \text{ if } n \text{ is even, }\\
	\lambda(\Delta_n^2)=2n & \text{ if } n \text{ is odd. }
	\end{cases}
	$$ 	
	As in the previous case, we divide it into cases where $n$ is even and odd:
	 
	 (a) {\textit{${\mathit{n}}$ is even}}: Let $k> {\frac {n} {2}}$.
	 
	  Suppose there is a $k^{\mathrm{th}}$ root, and thus a positive $k^{\mathrm{th}}$ root, for the generator of the center: let this positive root be $x$.
	 $$k \lambda(x) = \lambda(\Delta_n) = n.$$
	 Since $k> {\frac {n} {2}}$, this implies that $\lambda(x) < 2$. As before, we derive a contradiction. 
 
	 (b) {\textit{${\mathit{n}}$ is odd}}: Given $k>n$, if we assume an $k^{\mathrm{th}}$ root for the generator of the center in $I_2(n)$, it follows that there exists a $k^{\mathrm{th}}$ root in $I_2(n)^+$ from Proposition \ref{positive root}: call this positive root $x$. By similar arguments as before, we obtain the following equality:
	 $$k\lambda(x) = \lambda(\Delta_n^2).$$ 
	  It thus follows that $\lambda(x) < 2$.  This is a contradiction.
\end{proof}

\section{Elementary Equivalence}\label{elem}

\begin{theorem}\label{main theorem}
The elementary theories of irreducible Artin groups of finite type can be characterized as follows:
	\begin{itemize}
		\item[(1)]  $\operatorname{Th}(A_n) = \operatorname{Th}(A_m)$ if and only if $n = m$. 
		\item[(2)] $\operatorname{Th}(B_n) = \operatorname{Th}(B_m)$ if and only if $n = m$.
		\item[(3)] $\operatorname{Th}(D_n) = \operatorname{Th}(D_m)$ if and only if $n = m$.
		\item[(4)] For the family $I_{2}(n)$, the following holds:
		\begin{itemize}
		\item [(a)] If $m,n$ are odd, $\operatorname{Th}(I_2(n)) = \operatorname{Th}(I_2(m))$ if and only if $m=n$.
		\item [(b)]  If $m,n$ are even, $\operatorname{Th}(I_2(n)) = \operatorname{Th}(I_2(m))$ if and only if $m=n$. 
		\item [(c)] If $n$ is even, then for any $m>n$, $\mathrm{Th}(I_2(m)) \neq \mathrm{Th}(I_2(n))$.
		
\end{itemize}		
\end{itemize}
Each of the above results continues to hold in the $(\forall \exists \forall)^{1}$ fragment of the elementary theory.
\end{theorem}

\begin{proof}
Consider the following family of first order sentence, first introduced in Section 1:
	\[\Phi_k = \forall x. \exists y. \forall z.(\neg(xz = zy) \vee (x = y^k)).\] 
	$\Phi_k$ holds in a group precisely when every element in the center of the group has an $k^{\mathrm{th}}$ root; these are the sentences we will use to distinguish the elementary theories of the groups under consideration.
	
Whenever $m,n$ appear in this proof assume that $m>n$. 
	
	(1): The family $A_{k}$: Lemma \ref{center-root} implies that $\Delta_m^2$ has an $(m+1)^{\mathrm{th}}$ root in $A_{m}$. As $\operatorname{Z}(A_{m})$ is generated by $\Delta_m^2$, every element in $\operatorname{Z}(A_{m})$ has an $(m+1)^{\mathrm{th}}$ root. This implies that  $\Phi_{m+1}$ holds in $A_m$. As $m+1>n+1$, Lemma \ref{main lemma}(1) tells us that $\Delta_n^2$ does not admit an $(m+1)^{\mathrm{th}}$ root. This shows that $\Phi_{m+1}$ does not hold in $A_n$. Thus $\operatorname{Th}(A_n) \neq \operatorname{Th}(A_m)$.
	
	(2): The family $B_{k}$:  $\operatorname{Z}(B_m)$ is generated by $\Delta_m$. $\Delta_m$ has an $m^{\mathrm{th}}$ root in $B_m$ by Lemma \ref{center-root}, thus every element in $\operatorname{Z}(B_{m})$ has an $m^{\mathrm{th}}$ root. This implies that $\Phi_m$ holds in $B_m$. As $m>n$,  we see from Lemma \ref{main lemma}(2) that $\Delta_n$ does not admit an $m^{\mathrm{th}}$ root. Thus $\Phi_m$ does not hold true in $B_n$. 

	(3): The family $D_{k}$:  We subdivide the problem into following cases: 
	\begin{enumerate}
	\item[(i)]\emph{{\textit{Both $\mathit{m}$ and $ \mathit{n}$ are odd:}}}
	 As $\operatorname{Z}(D_m)$ is generated by $\Delta_m^{2}$, Lemma \ref{center-root} shows that  $\Phi_{2m-2}$ holds in $D_m$. As $m>n$, Lemma \ref{main lemma}(3(b)) shows that $\Phi_{2m-2}$ does not hold in $D_n$. 
	\item[(ii)]\emph{{\textit{Both $\mathit{m}$ and $\mathit{n}$ are even:}}}
     Here, $\operatorname{Z}(D_m)$ is generated by $\Delta_m$, by Lemma \ref{center-root}. Thus $\Phi_{m-1}$ holds in $D_m$. As $m>n$ it follows from Lemma \ref{main lemma}(3(a)) that $\Phi_{m-1}$ does not hold in $D_n$. 
    \item[(iii)] \emph{\textit One out of $\mathit{m}$ and $\mathit{n}$ is odd and the other is even:}
     Suppose first that $m$ is odd and $n$ is even. By examining parts 3(i) and 3(ii) of this proof, we note the following fact: the largest integer $l$ such that $\Phi_{l}$ holds in $D_{m}$ is even while the largest integer $j$ such that $\Phi_{j}$ holds in $D_{n}$ is odd. Thus $\operatorname{Th}(D_{m})\neq \operatorname{Th}(D_{n})$. 
     The same argument easily adapts to the case where $m$ is even and $n$ is odd. 
	\end{enumerate}
	
	(4): The family $I_{2}(n)$: As in the case of $D_n$, here we have sub-cases:
	\begin{itemize}
	\item[(a)]{{\textit{Both $\mathit{m}$ and $\mathit{n}$ are odd:}}} By Lemma \ref{center-root}, we know that $\Phi_n$ and $\Phi_m$ are hold in $I_2(n)$ and $I_2(m)$ respectively. From Lemma \ref{main lemma}(4(a)), it follows that $\Phi_m$ does not hold in $I_2(n)$.
	\item[(b)] {{\textit{Both $\mathit{m}$ and $\mathit{n}$ are even:}}}   By Lemma \ref{center-root}, we see that $\Phi_{\frac{n}{2}}$ and $\Phi_{\frac{m}{2}}$ are true in $I_2(n)$ and $I_2(m)$ respectively. From Lemma \ref{main lemma}(4(b)), it follows that $\Phi_{\frac{m}{2}}$ does not hold true in $I_2(n)$.
\item[(c)] If $n$ is even and $m > n$, then by parts 4(a) and 4(b) of this proof, we see that there is an integer $l > {\frac{n}{2}}$ such that $\Phi_{l}$ holds in $I_{2}(m)$. By an argument nearly identical to those used above, $\Phi_{l}$ does not hold in $I_{2}(n)$.
\end{itemize}
This completes the proof.
\end{proof}

\begin{corollary} Two braid groups have the same elementary theory if and only if they are isomorphic. The result continues to hold even if we only consider the $(\forall\exists\forall)^{1}$ fragment of the elementary theory.
\end{corollary}
\begin{proof}
The class $\{A_n\}_{n\in\mathbb{N}}$ represents braid groups: the group $A_{n}$ is the braid group on $n+1$ strands.
\end{proof}


\begin{thebibliography}{99}

	\bibitem[Bor01]{BGG} E. B{\"o}rger, E. Gr{\"a}del, and Y. Gurevich, 
	\textit{The classical decision problem}, Springer Science $\&$ Business Media, 2001.

	\bibitem[Bou02]{Bourbaki} N. Bourbaki,
	\textit{Elements of Mathematics, Lie Groups and Lie Algebras, Chapter 4-6}, Springer-Verlag, 2002. 
	
	\bibitem[BS72]{Artin-Cox} E. Brieskorn and K. Saito, \textit{Artin-Gruppen und Coxeter-Gruppen}, Inventiones Math., 17, (1972), 245-271.
	
	\bibitem[BM00]{BunM} E.I. Bunina, A.V. Mikhalev, \textit{Elementary equivalence of linear and algebraic groups}, Fundam. Appl. Math., 6 (3), (2000), 707-722.
	
	\bibitem[Bu01]{Bun} E.I. Bunina, \textit{Elementary equivalence of Chevalley groups}, Usp. Mat. Nauk, 155 (1), (2001), 157-158.
	
		\bibitem[BM04]{BM} E.I. Bunina and A.V. Mikhalev, \textit{Elementary properties of linear groups and related problems}, Journal of Mathematical Sciences, 123 (2), (2004), 3921-85.
	
	\bibitem[MKR08]{MKR} M. Casals-Ruiz, I. Kazachkov, and V. Remeslennikov, \textit{Elementary equivalence of right-angled Coxeter groups and graph products of finite abelian groups}, Bull. London Math. Soc., 42 (1), (2010), 130–136.
	
	\bibitem[CCCEHJR]{CCCEHJR} C. Coleman, R. Corran, J. Crisp, D. Easdown, R. Howlett, D. Jackson and A. Ram, \textit{A translation, with notes, of: E. Brieskorn and K. Saito, \textit{Artin-Gruppen und Coxeter-Gruppen}, Inventiones math. 17, (1972), 245-271.} 1996. \\ \url{http://www.ms.unimelb.edu.au/~ram/Resources/finbs5.12.97.pdf}
	
	
	\bibitem[DP99]{Deh-P} P. Dehornoy and L. Paris, \textit{Gaussian Groups and Garside Groups, Two Generalisations of Artin groups}, Proc. London Math. Soc., 79 (3),  (1999), 569-604.	
	
	\bibitem[Du95]{Dur} V.G. Durlev, \textit{On elementary theories of integer-valued linear groups,} Izv. Ross. Akad. Nauk. Ser. Mat., 59 (5), (1995), 41-58.
	
	\bibitem[FM12]{PMCG}  B. Farb and D. Margalit, \textit{A primer on mapping class groups}, Princeton Mathematical Series 49. Princeton University Press, 2012.
	
	 \bibitem[Gu63]{Gur} Y.S. Gurevich, \textit{Elementary properties of ordered Abelian groups}, Algebra Logika, 3 (1), (1964), 5-40.
	
	\bibitem[Hod06]{HW} W. Hodges, \textit{Model theory}, Encyclopedia of Mathematics and its Applications 42. Cambridge University Press, 1993.
	
	\bibitem[Ma61]{Mal} A.I. Maltsev, \textit{On elementary properties of linear groups}, Problems of Mathematics and Mechanics, Novosibirsk, (1961), 110-132.
	
	\bibitem[Ka63]{Kag} M.I. Kargapolov, \textit{On elementary theory of abelian groups}, Algebra Logika, 1 (6), (1963), 26-36.
	
	\bibitem[KM06]{KM} O. Kharlampovich and A. Myasnikov, \textit{Elementary theory of free non-abelian groups}, J. Algebra, 302 (2), (2006), 451-552.
	 
	 \bibitem[Ro60]{Rob} A. Robinson, E. Zakon, \textit{Elementary properties of ordered abelian groups}, Trans. Amer. Math. Soc., 96, (1960), 222-236.
	
	\bibitem[Sel06]{Sela} Z. Sela, \textit{Diophantine geometry over groups. VI. The elementary theory of a free group}, Geom. Funct. Anal., 16 (3), (2006), 707-730.
	
	\bibitem[Si04]{sibtame} H. Sibert, \textit{Tame Garside monoids}, J. Algebra, 281, (2004), 487-501.
	 
	 \bibitem[Si02]{sibgar} H. Sibert, \textit{Extraction of roots in Garside groups.}, Comm. Algebra, 30(6), (2002), 2915-2927.
	 
	 \bibitem[Sz55]{Szw} W. Szmielew, \textit{Elementary properties of Abelian groups}, Fund. Math. 41, (1955), 203-271. 
	 
	 

\end{thebibliography}
\end{document}